\documentclass[11pt,a4paper]{article}
\usepackage{url}
\usepackage[numbers]{natbib}
\usepackage[english]{babel}

\usepackage{amsmath, amssymb, amsthm}
\usepackage{amsfonts}
\usepackage{graphicx}

\newtheorem{thm}{Theorem}[section]
\newtheorem{lem}[thm]{Lemma}
\newtheorem{prop}[thm]{Proposition}

\theoremstyle{definition}

\newtheorem{rem}[thm]{Remark}

\numberwithin{equation}{section}

\newcommand{\W}{\widehat{W}_c}
\newcommand{\Fd}{\stackrel{D}{=}}
\newcommand{\fd}{\stackrel{d}{=}}

\begin{document}

\title{Geometric influence on the limiting behavior of\\
diffusion processes in one-sided Brownian environments\\
on disconnected fractal sets}

\author{
  Hiroshi Takahashi \\
  \small Keio University \\
  \small Yokohama, Kanagawa 223-8521, JAPAN \\
  \small \texttt{hstaka@fbc.keio.ac.jp}
}
\date{}

\maketitle

\begin{abstract}
We investigate diffusion processes on disconnected fractal sets 
in one-sided Brownian environments. 
On the real line, it is established that the process 
exhibits either diffusion or trapping, each occurring with probability $1/2$.
In this study, we demonstrate that for fractal sets, this behavior 
is governed by the geometric parameters $r$ 
(the reciprocal of the similitude ratio) 
and $N$ (the number of contraction mappings), 
which define the fractal structure.
Two distinct regimes emerge: 
a diffusive regime on the environment-free side 
and a localization regime on the side influenced by the environment. 
The transition between these regimes 
is determined by whether the random environment first hits 
the threshold $\log r$ or $-\log N$. 
Consequently, the probability of diffusion versus trapping 
is explicitly characterized by the Hausdorff dimension 
$d_f = \log N / \log r$. 
This result demonstrates that fractal geometry scales 
the limiting distributions and dictates 
the stochastic regime.

\vspace{1em}
\noindent \textbf{Keywords:} 
Diffusion process; Random environment; Fractal set; 
Localization; Hausdorff dimension
\end{abstract}

\section{Introduction}
Diffusion processes in random environments have been 
a central subject in stochastic analysis 
since the pioneering work of Brox \cite{br}. 
These processes constitute the continuous-time diffusion analogue 
of Sinai's random walk \cite{Si}. 
On the real line, it is well established that the random environment 
governs the limiting behavior, producing ultra-slow diffusion 
characterized by localization within environmental ``valleys.'' Kawazu, 
Suzuki and Tanaka \cite{KST} demonstrated that 
when the Brownian environment is restricted to 
the negative side of ${\mathbb R}$, 
the process exhibits a sharp dichotomy between 
standard diffusion and stochastic trapping.
For a comprehensive review of these developments 
and their broader context of stochastic 
processes in random environments, see \cite{S}.

In this paper, we extend the analysis from Euclidean spaces to 
disconnected self-similar fractal sets. 
Such sets are characterized by scaling properties, and 
diffusion processes on them exhibit propagation patterns that differ 
significantly from those in the Euclidean setting. 
We refer to \cite{F1} and \cite{F2} as foundational contributions 
in this field.
Among these fractal structures, disconnected sets such as Cantor sets present 
a unique framework for investigating anomalous transport. 
A central difficulty in this setting is that diffusion is anomalous even 
in the absence of environmental randomness. 
The scaling properties of fractal sets imply sub-diffusive behavior 
(see \cite{TT1}).
Recent studies demonstrate that diffusion on Cantor-like sets 
can undergo intricate transitions between sub-diffusive and super-diffusive 
regimes, depending on geometric constraints \cite{beyond}.
For instance, a fractal calculus approach has been applied 
to characterize such anomalous transport on fractal combs, 
underscoring the decisive role of the geometry 
\cite{Golmankhaneh2023}.

Beyond their theoretical significance, these disconnected fractal structures 
provide a framework for modeling disordered layered media.
In particular, layered media containing low-conductivity fractal 
inclusions have been shown to induce universal 
super-diffusive permeation \cite{layered}. 
This highlights that the spatial configuration of gaps and 
inclusions within the fractal geometry fundamentally 
governs the emergent macroscopic transport laws.

In our framework, we represent these structures 
through a self-similar speed measure $m_c$ defined 
on a disconnected fractal set 
and examine how the underlying geometry influences the transport dichotomy. 
When such a fractal structure is combined 
with a one-sided Brownian environment, a new question arises: 
How does the fractal geometry shape the competition 
between anomalous diffusion and stochastic trapping? 
In contrast to the $1/2$ probability split observed 
on the real line \cite{KST}, our analysis reveals that 
the limiting behavior is intrinsically governed 
by the Hausdorff dimension $d_f$ (Remark \ref{rem12}).

In Theorem \ref{t1}, we establish that the dichotomy 
between diffusion and stochastic trapping is 
governed by the random environment 
and by the geometric properties of fractal sets, 
namely the reciprocal of the similitude ratio $r$ 
and the number of contraction mappings $N$. 
When the ``hill'' in the one-sided random environment 
is sufficiently high, the process is repelled towards 
the environment-free side, thereby recovering the anomalous diffusive behavior 
intrinsic to the fractal. Conversely, when the ``hill'' 
is low and the environmental fluctuations are sufficiently deep, 
the process localizes within the environment's ``valley.'' 
The thresholds for these heights and depths are explicitly 
determined by the interplay between $r$ and $N$. 
This, in turn, indicates that the fractal dimension $d_f$ serves as 
a critical bifurcation index for the phase transition between diffusion 
and trapping. Consequently, the geometric properties of fractal sets are not 
merely scaling factors 
of anomalous diffusion but rather decisive determinants of 
the fundamental stochastic regime.
The following subsections provide a detailed mathematical 
exposition of our model and main results.

\subsection{One-dimensional cases}

Let ${\mathbb W}$ denote the space of continuous functions 
$W: {\mathbb R} \to {\mathbb R}$ with $W(0)=0$. 
Let $Q$ be the Wiener measure on ${\mathbb W}$ 
such that $\{W(x), x \geq 0\}$ and $\{W(-x), x \geq 0\}$ 
are independent one-dimensional Brownian motions starting 
at the origin. 
The pair $(W,Q)$ denotes a Brownian environment on ${\mathbb R}$. 
Given such an environment, we consider the diffusion process 
$X(t, W)$ starting at $0$, with generator 
\[
L_W=\frac{1}{2} \exp(W(x)) \frac{d}{dx} 
\left( \exp(-W(x)) \frac{d}{dx} \right). 
\]
Brox \cite{br} established that the distribution of 
$(\log t)^{-2} X(t)$ is tight, 
thereby demonstrating that the diffusion process exhibits 
ultra-slow diffusive behavior.

In \cite{KST}, Kawazu, Suzuki and Tanaka analyzed the case 
where $W(x)=0$ for $x \geq 0$ and 
$\{W(-x), x \geq 0, Q\}$ constitutes a Brownian environment. 
In this setting, $X(t, W)$ is referred to as a one-dimensional 
diffusion process in a one-sided Brownian environment 
(or with a one-sided Brownian potential). 
Their results demonstrate that the limiting behavior of 
$\{X(t, W), t \geq 0\}$ is diffusive (i.e., a limit distribution 
exists under Brownian scaling) with probability $1/2$, 
and ultra-slow diffusive with the remaining probability $1/2$. 
Building on this study, Kawazu and Suzuki obtained more precise results 
for this model in \cite{KS}. Moreover, the model has been extended 
to environments governed 
by strictly $\alpha$-stable L\'{e}vy processes for $\alpha \in (0,2]$ 
in \cite{STT}.

The limit theorems in \cite{STT} are derived 
by employing the general theory of one-dimensional bi-generalized 
diffusion processes established by Ogura \cite{Og}, 
together with the refined limit theorem of Tanaka \cite{T}. 
A bi-generalized diffusion process is a one-dimensional Markov process 
characterized by a non-decreasing scale function $s(x)$ 
and a non-negative speed measure $m(dx)$. 
Since its introduction, this framework has been extensively developed 
to analyze diverse diffusion processes in both theoretical 
and applied contexts. 
For instance, Ogura \cite{Og} provided an alternative proof of Brox's 
theorem as an early application. Its versatility is further illustrated 
by subsequent studies: 
Iizuka and Ogura \cite{IO} applied it to genetic models, and 
Takemura, Tomisaki and Iizuka investigated bi-generalized processes 
on finite intervals in \cite{TTI}. More recently, Takahashi and Tamura applied 
this approach to Brox-type diffusions on disconnected fractal sets 
in \cite{TT2}.

Building on these developments, the present study utilizes 
the limit theorems for bi-generalized diffusion processes 
to establish our main results on disconnected self-similar sets.

\subsection{Diffusion processes on disconnected fractal sets}

Based on the framework in \cite{TT1} and \cite{TT2}, 
we provide a precise definition of disconnected self-similar fractal sets 
and diffusion processes associated with them.

Let $r>1$ and $\Phi=\{\varphi_1, \varphi_2, \ldots, \varphi_N\}$ 
denote a family of $r$-similitudes on $[0,1]$. 
We assume that 
(i)\ $\varphi_1(x)=x/r$, 
(ii)\ $\varphi_N(x)=x/r+(1-1/r)$, and 
(iii)\ $\varphi_i([0,1]) \cap \varphi_j([0,1]) = \emptyset$ for $i \neq j$. 
These assumptions imply $N<r$, and without loss of generality 
we may assume $\varphi_i(1) < \varphi_{i+1}(0)$
for $i=1,2, \ldots, N-1$. 
Under these conditions, 
there exists a uniquely determined compact self-similar fractal set 
$C_0 \subset [0,1]$ associated with $\Phi$. 
By extending $C_0$ symmetrically to the whole real line 
via scaling and reflection, we obtain an unbounded disconnected self-similar 
fractal set $C \subset \mathbb{R}$. 
Corresponding to the set $C_0$, there exists 
a unique normalized self-similar measure. 
This extends symmetrically to an infinite self-similar measure $m_c$ 
on $\mathbb{R}$ satisfying:
\begin{align}\label{scale}
m_c(A)=\frac{1}{N^n} m_c(r^n A)
\end{align}
for any Borel set $A \subset {\mathbb R}$ and $n \in {\mathbb Z}$. 
Although both the underlying space $C$ 
and the measure $m_c$ are symmetrically defined 
on the positive and negative sides, the random 
environment $\W$ is introduced only on the negative side,
as specified later in \eqref{environment}.
For simplicity, we also regard $m_c(x)$ as a self-similar function defined by 
\begin{align}\label{m_c}
m_c(x):=
\begin{cases}
m_c([0,x]), & x \geq 0, \\
-m_c([x, 0]), & x<0. 
\end{cases}
\end{align}
Using $m_c$, we define the generator 
\[
\frac{1}{2} \frac{d}{dm_c} \frac{d}{dx}
\]
and consider the associated one-dimensional generalized diffusion process 
starting at the origin, denoted by $\{B_{c}(t), t \ge 0\}$. 
We refer to $\{B_{c}(t)\}$ as the diffusion process on $C$. 
We set 
$\Omega=\{\omega \in \mathcal{C}([0,\infty) \to {\mathbb R}): \omega(0)=0\}$ 
and let $P$ be the Wiener measure on $\Omega$. 
Let $B(t)$ denote the value of a function $\omega \in \Omega$ at time $t$. 
Then, $B_c(t)$ is given by a scale change and a time change of $B(t)$, 
where the time change is determined by $m_c$, in accordance with the theory of 
generalized diffusion processes (\cite{IM}).   
We set
\begin{align*}
L(t, x) &:= \lim_{\varepsilon \to 0} {\frac{1}{\varepsilon}} 
\int_0^t 1_{[x, x+\varepsilon)}(B(s)) \, ds\quad 
\text{(the local time of $B(t)$ at $x$)}, \\ 
A_c(t) &:= \int_{\mathbb R} L(t, x)\, m_c(dx), \\
B_c(t)&:= B(A_c^{-1}(t)). 
\end{align*}
In \cite{TT1}, the process $\{B_c(t)\}$ is constructed as 
the limit of a sequence of appropriately scaled random walks, 
and homogenization problems 
on disconnected fractal sets are investigated. 
Moreover, $\{B_c(t)\}$ satisfies the following scaling relation: 
for each $n \in {\mathbb Z}$, 
\begin{align}\label{cant}
\{B_c(t), t \geq 0\} \fd \left\{ \frac{1}{r^n} B_c((rN)^nt), 
t \geq 0 \right\}, 
\end{align}
where $\fd$ denotes equality in distribution with respect to $P$. 
The scaling parameter $rN$ can equivalently be expressed as $r^{d_f+1}$, 
where $d_f=\log N/\log r$ is the Hausdorff dimension of $C$.

We next consider diffusion processes in Brownian environments 
defined on disconnected fractal sets, as studied in \cite{TT2}. 
Given $(W,Q)$ and $m_c(x)$, we introduce 
\begin{align}\label{ENV}
W_c(x) = W(m_c(x))
\end{align}
and refer to $(W_c, Q)$ as the Brownian environment on $C$. 
For a given $W_{c}$, we consider a generalized one-dimensional 
diffusion process $\{X(t,W_{c}), t\ge 0\}$ starting at the origin 
whose generator is given by
\begin{align}\label{gene1}
L_{W_c}=\frac{1}{2} \exp(W_c(x)) \frac{d}{dm_c} 
\left( \exp(-W_c(x)) \frac{d}{dx} \right).
\end{align}
The diffusion process $\{X(t,W_{c})\}$ is referred to as a diffusion process 
in a Brownian environment on $C$. We assume that 
$\{W_{c}(x), x\in \mathbb{R}\}$ and the standard Brownian motion 
$\{B(t), t\ge 0\}$ employed in the construction of $B_c(t)$ are independent. 
Under this assumption, $\{X(t,W_{c})\}$ is obtained via a scale change 
and a time change of \(B_{c}(t)\) determined by \(W_{c}\) as follows: 
\begin{align*}
S(x) &:= \int_0^x \exp(W_c(y))\, dy, \\
A(t) &= A(t, B_c):=\int_{\mathbb R} \exp\left(-2W_c(S^{-1}(x))\right)
L(t,x)\, m_c(dx)\\
&= \int_0^t \exp\left( -2W_c(S^{-1}(B_c(s))) \right)\, ds, \\
X(t, W_c) &:= S^{-1}(B_c(A^{-1}(t))).
\end{align*}

To characterize the limiting behavior of the process, 
we first introduce the notion of a valley for $W$ 
following \cite{Ke}, and then adapt it to the fractal setting. 
For $W \in {\mathbb W}$, we define 
\begin{align*}
W^{\#}(x) &:= 
\begin{cases}
W(x)-\min\{W(y): 0 \leq y \leq x\}, & x \geq 0, \\
W(x)-\min\{W(y): x \leq y < 0\}, & x < 0, 
\end{cases} \\
\ell^{+}(u) &:= \inf \{x>0: W^{\#}(x) = u\}, \\
\ell^{-}(u) &:= \sup \{x<0: W^{\#}(x) = u\}.
\end{align*}
Using this construction, we set
\begin{align}\label{V}
V^{+}(u) &:= \min\{W(y):\ 0 \leq y \leq \ell^{+}(u)\}, \nonumber \\ 
V^{-}(u) &:= \min\{W(y):\ \ell^{-}(u) \leq y \leq 0\},
\end{align}
and define $b^{+}(u)$ and $b^{-}(u)$ as 
the points satisfying the following: 
\begin{align}
W(b^{+}(u)) &= V^{+}(u), \nonumber \\
W(b^{-}(u)) &= V^{-}(u). 
\end{align}
Note that for each fixed $u>0$, the points $b^+(u)$ and $b^-(u)$ are 
determined uniquely for $Q$-almost all environments 
(see \cite[Lemma 2.2]{br}). 
We define 
\begin{align}\label{height}
M^{+}(u) &:= \max\{W(y):\ 0 \leq y \leq b^{+}(u)\}, \nonumber \\ 
M^{-}(u) &:= \max\{W(y):\ b^{-}(u) \leq y \leq 0\}, 
\end{align}
which represent the ``heights of hills'' 
of the environment associated with the process. 
The ``bottom of the valley'' for \(W\), denoted by \(b(u,W)\), 
is given by
\begin{align}\label{BOTT}
b(u,W):= 
\begin{cases}
b^{+}(u)& \text{if } \max\{M^{+}(u),(V^{+}(u)+u)\} 
			< \max\{M^{-}(u),(V^{-}(u)+u)\},\\ 
b^{-}(u)& \text{if } \max\{M^{+}(u),(V^{+}(u)+u)\}
			> \max\{M^{-}(u),(V^{-}(u)+u)\},
\end{cases}
\end{align}
and we refer to $u$ as the ``depth of the valley.''

Finally, we define the corresponding bottom in the fractal setting. 
The point $b_{c}(u)=b_{c}(u,W_{c})$ is given by 
\begin{align}\label{b_c}
b_c(u) := m_c^{-1}(b(u, W)),
\end{align}
where the inverse function $m_c^{-1}$ is defined as follows:
\begin{align*}
m_c^{-1}(y) := \begin{cases}
\inf\{ x \in C : m_c(x) \ge y \}, & y \geq 0, \\
\sup\{ x \in C : m_c(x) \le y \}, & y < 0.
\end{cases}
\end{align*}
This definition guarantees that $b_c(u)$ represents the point 
in $C$ that is uniquely determined 
as the bottom of the valley in the fractal domain.

We define a probability measure ${\mathcal P}$ on 
$\Omega \times {\mathbb W}$ 
such that ${\mathcal P}= P \otimes Q$. 
Let $\{X(t), t \geq 0\}$ denote a process defined 
on the probability space $(\Omega \times {\mathbb W}, {\mathcal P})$. 
In \cite{TT2}, the following limit theorem is established: 
\begin{thm}[\cite{TT2}]\label{TT2}
The finite-dimensional distributions of the process 
$\displaystyle{
\left\{ 
r^{-2n} X(\exp(N^n u)), u > 0 
\right\}}$ 
under ${\mathcal P}$ converge as $n \to \infty$ 
to the corresponding finite-dimensional distributions 
of the process $\{b_c(u), u>0\}$ under $Q$. 
\end{thm}
While Theorem \ref{TT2} establishes the ultra-slow behavior 
and localization of the process under a logarithmic time scaling $\exp(N^n u)$ 
in two-sided environments, the restriction to a one-sided environment 
introduces a subtler asymptotic balance. 
Building on this framework, the present study focuses on the one-sided case, 
where the environment is confined to the negative side of $\mathbb{R}$. 
Our objective is to elucidate the role of 
the fractal parameters $r$ and $N$ as decisive determinants 
of the dichotomy between diffusive and trapping regimes.

\section{Model and result}

We define 
\begin{align}\label{environment}
\W(x):=
\begin{cases}
0,& x >0, \\ 
W_c(x),& x \leq 0,
\end{cases}
\end{align}
and refer to $(\W, Q)$ as a one-sided Brownian environment 
on the disconnected fractal set $C$. 
The environment satisfies the semi-self-similar 
scaling relation: for each $n \in {\mathbb N}$, 
\begin{align}\label{scale2}
\left\{ \W(x), x \leq 0 \right\} \Fd 
\left\{ \frac{1}{N^n} \W(r^{2n} x), x \leq 0 \right\},
\end{align}
where $\Fd$ denotes equality in distribution with respect to $Q$. 
We also define 
\[
\W^{(n)}(x) := N^{-n}\W (r^{2n}x).
\]

To examine the interplay between the geometric structure of 
fractal sets and random environments, we focus on the case where 
the depth of the valley is specified by $u = \log rN$. 
In the setting of a one-sided random environment $\W$, 
the bottom of the valley $b_c(\log rN, \W)$ lies on the negative side 
of ${\mathbb R}$ for $Q$-almost all random environments, 
as established in \eqref{V}--\eqref{b_c}. 
For simplicity, we write $V$ and $M$ for $V^-$ and $M^-$ 
defined in \eqref{V} and \eqref{height}, respectively. 
For each $\W$, we define 
\begin{align}
\rho^+ &=\rho^+(\W):=\sup 
\left\{x<0: \W(x)>\log r\right\}, \label{rho} \\
\rho^- &=\rho^-(\W):=\sup 
\left\{x<0: \W(x)<-\log N\right\}. \label{rho2}
\end{align}
Using these, we classify one-sided Brownian environments 
as follows: 
\begin{align*}
{\mathbb A}&=\left\{ \W: \rho^+ > \rho^- \right\},\quad 
{\mathbb A}_n=\left\{ \W: \W^{(n)} \in {\mathbb A} \right\},\\ 
{\mathbb B}&=\left\{ \W: \rho^+ < \rho^- \right\},\quad 
{\mathbb B}_n=\left\{ \W: \W^{(n)} \in {\mathbb B} \right\}. 
\end{align*}
Intuitively, $\mathbb{A}$ represents the set of environments 
in which the ``hill'' repels the process, 
whereas $\mathbb{B}$ corresponds to the set of those 
in which the ``valley'' confines it.

We set 
\begin{align*}
\hat{m}_c(x):=
\begin{cases}
m_c(x),& x \geq 0, \\ 
0,& x < 0. 
\end{cases}
\end{align*}
Let $\widehat{B}_c(t)$ denote a one-dimensional generalized diffusion 
process starting at the origin, corresponding to the generator 
\begin{align}\label{gene}
\frac{1}{2}\frac{d}{d\hat{m}_c(x)} \frac{d}{dx}. 
\end{align}
This process evolves on the positive side of $C$ with a reflecting boundary 
at the origin. 
For a fixed integer $d \geq 2$ and $0 < \varepsilon <1$, we then define 
\begin{align}
T_{d, \varepsilon}:=\{(t_1, \ldots, t_d) \in (0, \infty)^d: 
\ \varepsilon \leq t_1< \cdots < t_d \leq 1/\varepsilon,\ 
t_k-t_{k-1} \geq \varepsilon,\ 
k=2,\ldots, d\}.
\end{align}
The following theorem establishes the phase transition of the process 
$X(t, \W)$, arising from the interaction between the environmental randomness 
and the geometric parameters of the fractal set.

\begin{thm}\label{t1} For any integer $d \geq 2$ and $\varepsilon \in (0,1)$, 
the following hold:

$(1)$ 
For any $a_k > 0,\ k=1, \dots, d$, the joint probability
\begin{align}\label{2-1}
\int P
\left\{ 
\frac{1}{r^{N^n}} X((rN)^{N^n} t_1, \W) > a_1,\ \ldots,\ 
\frac{1}{r^{N^n}} X((rN)^{N^n} t_d, \W) > a_d
	\right\}\, Q(d\W \mid \mathbb{A}_n)
\end{align}
converges to 
\[
P \left\{ \widehat{B}_c(t_1)>a_1,\ \ldots,\ 
\widehat{B}_c(t_d)>a_d \right\}
\]
as $n \to \infty$ uniformly over 
$(t_1, \ldots, t_d) \in T_{d,\varepsilon}$, 
where $\widehat{B}_c(t)$ 
denotes the generalized diffusion process generated by 
\eqref{gene}. 

$(2)$  For any $\varepsilon_k > 0,\ k=1, \dots, d$, 
the probability 
\begin{align}\label{(2)}
\int P\left\{ 
\left| 
\frac{1}{r^{2n}} X((rN)^{N^n} t_k, \W) - b_c(\log rN, \W^{(n)})
\right| 
< \varepsilon_k,\ k=1, \ldots, d \right\}\, 
Q(d\W \mid \mathbb{B}_n)
\end{align}
converges to $1$ as $n \to \infty$ uniformly over 
$(t_1,\ \ldots, t_d) \in T_{d,\varepsilon}$. 
\end{thm}

\begin{rem}\label{rem12}
In the Euclidean setting \cite{KST}, 
the process enters the diffusive and trapping regimes 
with equal probability $1/2$, reflecting the inherent symmetry of 
the environment. Our results reveal that on fractal sets, 
this balance is disrupted by the geometric parameters $r$ and $N$. 
Specifically, the threshold conditions \eqref{rho} and \eqref{rho2} 
depend on $\log r$ and $\log N$, yielding the diffusive regime 
$\mathbb{A}$ with probability 
\[
Q(\mathbb{A}) = \frac{\log N}{\log r + \log N} = \frac{d_f}{1+d_f},
\]
whereas the trapping regime $\mathbb{B}$ occurs with probability 
$Q(\mathbb{B})=1/(1+d_f)$. 
These formulas demonstrate that the Hausdorff dimension 
$d_f = \log N / \log r$ serves as the decisive selector 
of the stochastic regime. 
Notably, when $d_f=1$, these probabilities recover 
the $1/2$ split observed in the Euclidean case.
\end{rem}

\section{Proof of Theorem \ref{t1} (1)}

To prove the assertion, we consider scaled one-sided 
Brownian environments. For each $n \in {\mathbb N}$, we define 
\begin{align}\label{h_n}
h_n:=r^{N^n-2n}
\end{align}
and introduce the scaling operator $F_n$ such that 
\begin{align}
(F_n \W)(x):=N^n \W(h_n x),\quad x \in {\mathbb R}.
\end{align}
Note that $(F_n \W^{(n)})(x)=\W(r^{N^n}x)$. 
The subsequent lemma establishes the scaling property of 
diffusion processes in a one-sided Brownian environment on $C$: 

\begin{lem}[\cite{TT2}]\label{lem_scale}
For each $n \in {\mathbb N}$, it holds that 
\begin{align}\label{lem31}
\left\{ \frac{1}{r^{N^n}} X((rN)^{N^n} t, \W),t \geq 0 \right\} 
\fd 
\left\{X(t, F_n \W^{(n)}), t \geq 0 \right\}. 
\end{align}
\end{lem}

Using this property and the semi-self-similarity of the environments 
\eqref{scale2}, we analyze the convergence of the probability 
\begin{align}\label{Q(A)}
\frac{1}{Q({\mathbb A})} \int_{\mathbb A}
P \left\{
X(t_1, F_n \W) > a_1,\ \ldots,\ 
X(t_d, F_n \W) > a_d
\right\}\, Q(d\W)
\end{align}
as $n \to \infty$ instead of \eqref{2-1} in Theorem \ref{t1} (1). 
For a fixed $F_n \W$, we consider the generator 
\begin{align}
L_{F_n \W}=\frac{1}{2} \exp\left(F_n \W \right) \frac{d}{dm_c} 
\left( \exp\left(-F_n  \W \right) \frac{d}{dx} \right).
\end{align}
This generator implies that the scale function $\hat{S}_n(x)$ 
and speed measure $\hat{m}_n(dx)$ 
associated with $X(t, F_n \W)$ are given by
\begin{align}
\hat{S}_n(x) &:=\int_0^x \exp\left( (F_n \W)(y) \right)dy, \\ 
\hat{m}_n(dx) &:=2 \exp\left( -(F_n \W)(x) \right) m_c(dx). 
\end{align}
Note that $\hat{S}_n(x)=x,\ x \geq 0$. 
For $x \in {\mathbb R}$, we define 
\begin{align}\label{3_8} 
\hat{m}_n(x) & := 
\begin{cases} 
\hat{m}_n([0,x]), & x \geq 0, \\ 
-\hat{m}_n([x, 0]), & x<0, 
\end{cases} 
\end{align}%
{that is, we adopt the same convention as used for $m_c(x)$ in 
\eqref{m_c}}. 
Using these, we set 
\begin{align}\label{3_9}
\hat{M}_n(x) := \hat{m}_n \circ \left( \hat{S}_n \right)^{-1}(x) 
= \sup \left\{ \hat{m}_n(y):\ \hat{S}_n(y) \leq x \right\}.
\end{align}
Then, we obtain the following proposition: 

\begin{prop}\label{prop}
For $Q$-almost all $\W \in {\mathbb A}$, it holds that 
\begin{align}
\lim_{n \to \infty} \hat{M}_n(x) = 2\hat{m}_c(x),\quad x \in {\mathbb R}.
\end{align}
\end{prop}

\begin{proof}
We adopt the strategy employed in the proof of Lemma 5.1 
in \cite{STT}. On the negative side, we observe 
\begin{align*}
\hat{S}_n(h_n^{-1} z)&=\int_0^{h_n^{-1}z} 
\exp\left( N^n \W(h_n y)\right)\ dy \\
&= \frac{1}{h_n} \int_0^z \exp\left(N^n \W(u)\right) \, du\\
&= -\frac{r^{2n}}{r^{N^n}} \int_z^0 
\exp \left( N^n \W(u)\right)\, du,
\end{align*}
where $u=h_n y$. 
By the classical Laplace method, 
the asymptotic behavior of this integral is governed by 
the supremum of the environment $\W$ over the interval of integration. 
Specifically, we obtain 

\[
\lim_{n \to \infty} \frac{1}{N^n} \log \int_z^0 
\exp \left(N^n \W (u)\right)\, du 
= \sup\left\{\W(x):\ z < x <0\right\}. 
\]
Combining this with the scaling factor $h_n$ defined in 
\eqref{h_n}, we establish the following convergences:

(I)\ For any $\varepsilon >0$ satisfying $\rho^+ + \varepsilon <0$, 
\[
\lim_{n \to \infty} \hat{S}_n(h_n^{-1}(\rho^++\varepsilon))=0 \quad 
{\text{for }Q\text{-almost all }\W \in \mathbb{A}}. 
\]

(II)\ For any $\varepsilon >0$, 
\[
\lim_{n \to \infty} \hat{S}_n(h_n^{-1}(\rho^+-\varepsilon))=-\infty \quad 
{\text{for }Q\text{-almost all }\W \in \mathbb{A}}. 
\]

\noindent 
The preceding assertions imply that, for any fixed $y<0$ and sufficiently 
small $\varepsilon >0$, there exists $n_0$ such that for any 
$n > n_0$
\begin{align*}
\hat{S}_n(h_n^{-1}(\rho^+ - \varepsilon)) < y < 
\hat{S}_n(h_n^{-1}(\rho^+ + \varepsilon)).
\end{align*}
Consequently, for any $y<0$, we obtain
\begin{align}\label{rho_lim}
\lim_{n \to \infty} h_n \left( \hat{S}_n \right)^{-1}(y)=\rho^+.
\end{align}
For a fixed $y_1<0$, we define 
$z_{n,1} := \left( \hat{S}_n \right)^{-1}(y_1)$. 
{Using the property \eqref{scale}}, we obtain 
\begin{align*}
\hat{M}_n(y_1) &= \hat{m}_n(z_{n,1}) \\ 
&= -2 \int_{\left( \hat{S}_n \right)^{-1}(y_1)}^0
\exp\left( -N^n \W(h_n y)\right ) \, m_c(dy) \\ 
&=- \frac{2N^{2n}}{N^{N^n}} 
\int_{h_n \left( \hat{S}_n\right)^{-1}(y_1)}^0 
\exp\left( -N^n \W (u)\right) \, m_c(du),
\end{align*}
where $u=h_n y$. 
Combining this with \eqref{rho_lim}, 
it follows that for any $\varepsilon > 0$, 
there exists $n_0$ such that for any $n > n_0$, 
\[
\frac{2N^{2n}}{N^{N^n}} 
\int_{h_n \left( \hat{S}_n\right)^{-1}(y_1)}^0 
\exp\left( -N^n \W (u)\right) \, m_c(du)
\leq 
\frac{2N^{2n}}{N^{N^n}} 
\int_{\rho^+ - \varepsilon}^0 \exp \left(-N^n\W(y)\right) \, m_c(dy).
\]
{Since $\rho^+ > \rho^-$ for any $\W \in {\mathbb A}$, 
it follows that $\inf\left\{\W(x): \rho^+ <x<0 \right\} > -\log N$. 
Moreover, by the continuity of $\W(x)$, we can choose $\varepsilon > 0$ 
sufficiently small such that 
$\inf\left\{\W(x): \rho^+ - \varepsilon \le x \le 0\right\} > -\log N$.
Consequently, applying the classical Laplace method yields
\[\lim_{n \to \infty}
\frac{2N^{2n}}{N^{N^n}} \int_{\rho^+ - \varepsilon}^0 
\exp \left( -N^n \W(y)\right)\, m_c(dy) = 0.
\]}
This implies that
\[
\lim_{n \to \infty} \hat{M}_n(x)=0, \quad x<0.
\]

On the non-negative side, which is independent of the environment, we have 
$\hat{S}_n(x) = x$ and $\hat{m}_n(dx) = 2 m_c(dx)$. 
Consequently, it follows that $\lim_{n \to \infty}\hat{M}_n(x) = 
2\hat{m}_c(x)$ for all $x \geq 0$.

By combining these two cases, we deduce that 
$\lim_{n \to \infty} \hat{M}_n(x) = 2\hat{m}_c(x)$ 
holds for all $x \in \mathbb{R}$.
\end{proof}

The convergence established in Proposition \ref{prop} provides 
a clear probabilistic interpretation 
of the diffusive regime. The transformation $\hat{M}_n$ 
corresponds to the speed measure of the process $X(t, F_n \W)$, 
normalized by the scale function $\hat{S}_n$. 
The vanishing of $\hat{M}_n(x)$ for $x < 0$ in the limit 
indicates that the ``residence time'' of the process 
in the negative environment becomes negligible. 
Consequently, the limiting process $\widehat{B}_c$ is confined
to the non-negative region $J_0 := [0, \infty)$, 
where the environment is absent.

\begin{proof}[Proof of Theorem \ref{t1} (1)]

We follow the proofs of Theorem 3.2 (1) and Theorem 5.3 in \cite{STT}. 
Let $\hat{q}_n(t,x,y)$ and $\hat{q}(t,x,y)$ denote 
the transition densities associated with the generalized 
diffusion processes corresponding to $(x, \hat{M}_n)$ 
and $(x, 2\hat{m}_c)$, respectively. 
By Proposition \ref{prop} and Proposition 5.1 in \cite{Og}, 
we deduce that $\hat{q}_n(t,x,y)$ converges to $\hat{q}(t,x,y)$ 
uniformly on compact subsets of $(0, \infty) \times {\mathbb R} 
\times {\mathbb R}$ as $n \to \infty$. 
Consequently, for $Q$-almost all $\W \in {\mathbb A}$, and 
for any integer $d \geq 2$, $\varepsilon \in (0,1)$, and 
$f_k \in C_0(J_0), k=1,\ldots, d$, it follows that 
\begin{align}\label{fd}
\lim_{n \to \infty} 
E\left[ \prod_{k=1}^d f_k \left( X(t_k, F_n\W) \right) \right] 
= 
E\left[ \prod_{k=1}^d f_k \left( \widehat{B}_c(t_k) \right) \right]
\end{align}
uniformly over $(t_1, \ldots, t_d) \in T_{d, \varepsilon}$.

Here, \eqref{fd} implies that for any $0 < a_k < a_k',\ k=1, \ldots, d$, 
\begin{align}\label{Proof_1}
\lim_{n \to \infty}
P \left\{ a_k < X(t_k, F_n\W) < a_k',\ k=1, \ldots, d \right\} 
=
P \left\{ a_k < \widehat{B}_c(t_k) < a_k',\ k=1, \ldots, d \right\}
\end{align}
uniformly over $(t_1, \ldots, t_d) \in T_{d, \varepsilon}$ 
for $Q$-almost all $\W \in {\mathbb A}$. 
Following the argument employed in the proof of 
Theorem 6.1 in \cite{Og} (specifically, establishing (6.6) therein), 
we obtain 
\begin{align}\label{Proof_2}
\lim_{a \to \infty} \limsup_{n \to \infty} 
\sup_{\varepsilon \leq t \leq 1/\varepsilon} 
P \left\{ X(t, F_n\W)>a \right\} =0.
\end{align}
By the definition of $\hat{m}_c(x)$, it follows that
\begin{align}\label{Proof_3}
\lim_{a \to \infty} 
\sup_{\varepsilon \leq t \leq 1/\varepsilon} 
P \left\{ \widehat{B}_c(t)>a \right\} =0.
\end{align}
From the convergences \eqref{Proof_1}--\eqref{Proof_3}, 
the claim follows. 
Specifically, these convergences, together with the tightness 
of the distributions, justify the application of 
the dominated convergence theorem 
under the measure $Q(\cdot \mid \mathbb{A}_n)$. This establishes the proof. 
\end{proof}

\section{Proof of Theorem \ref{t1} (2)}
To derive the limiting distribution, we introduce the scaling operator 
$G_n$ for each $n \in \mathbb{N}$ defined by
\[
(G_n \W)(x):=N^n \W(x),\quad x \in {\mathbb R}. 
\]
We observe that $(G_n \W^{(n)})(x)=\W(r^{2n}x)$. 
For each $n \in \mathbb{N}$, we define 
\begin{align}
j_n:=(rN)^{N^n-2n}. 
\end{align}
Then, the following scaling relation holds:

\begin{lem}\label{sc_j}
For each $n \in {\mathbb N}$, it holds that 
\begin{align}\label{lem41}
\left\{ \frac{1}{r^{2n}} X((rN)^{N^n} t, \W), t \geq 0 \right\} 
\fd 
\left\{X(j_n t, G_n \W^{(n)}), t \geq 0 \right\}. 
\end{align}
\end{lem}

{\begin{proof}
Lemma 4.1 in \cite{TT2} implies that 
\begin{align*}
\left\{ \frac{1}{r^{2n}} X((rN)^{2n}t, \W), t \geq 0 \right\} 
\fd \left\{X(t, G_n \W^{(n)}), t \geq 0 \right\}. 
\end{align*}
Replacing $t$ by $j_n t$ in the preceding equation yields 
the scaling relation \eqref{lem41}.
\end{proof}}

Analogously to \eqref{Q(A)}, we examine 
the convergence of the probability 
\begin{align}
\frac{1}{Q({\mathbb B})} \int_{\mathbb B} 
P \left\{ 
\left| X(j_n t_k, G_n \W) -b_c(\log rN, \W) \right| < \varepsilon_k,\ 
k=1, \ldots, d \right\} 
Q(d \W)
\end{align}
as $n \to \infty$ instead of \eqref{(2)} in Theorem \ref{t1} (2). 
For a fixed $\W$, we set 
\[
b_c = b_c(\W) := b_c(\log rN, \W). 
\]
We then introduce the shifted process and 
the self-similar function associated with $b_c$ defined by 
\begin{align}\label{Y_n}
Y_n(t, G_n\W)&:=X(j_nt, G_n\W)-b_c, \\
m_c^{b_c}(x)&:= m_c(x+b_c)-m_c(b_c). \nonumber
\end{align}
The generator of $Y_n(t, G_n \W)$ is given by the following: 
\begin{align}\label{geney}
L_{G_n \W}=\frac{j_n}{2} \exp 
\left( (G_n \W)(x+b_c)-(G_n \W)(b_c) \right) 
\frac{d}{dm_c^{b_c}}
\left( 
\exp 
\left( -\{(G_n \W)(x+b_c)-(G_n \W)(b_c)\} \right) \frac{d}{dx}
\right).
\end{align}

To establish Theorem \ref{t1} (2), 
we partition the set $\mathbb{B}$ into two subsets according 
to the height $M$: 
\begin{align*}
{\mathbb B}_1&:=\left\{ \W \in {\mathbb B}: M < V + \log rN \right\},\\ 
{\mathbb B}_2&:=\left\{ \W \in {\mathbb B}: M > V + \log rN \right\}, 
\end{align*}
{where, as introduced in Section 2, 
$V$ and $M$ are used as shorthand for 
$V^-$ in \eqref{V} and $M^-$ in \eqref{height}, respectively.}

\subsection{Case of ${\mathbb B}_1$}

Taking $\delta \in (0, (-\log N - V))$, we set 
\begin{align*}
\ell_c&=\ell_c(\W) 
:= \sup \left\{ y<b_c:\ \W(y)-\W(b_c)=\log rN \right\}, \\
\tilde{\ell}_c &= \tilde{\ell}_c(\W) 
:= \sup \left\{ y < b_c:\ \W(y) - \W(b_c) = \delta \right\}, \\
\tilde{r}_c &=\tilde{r}_c(\W)  
:= \inf \left\{y > b_c:\ \W(y) - \W(b_c) = \delta \right\}. 
\end{align*}
From the generator \eqref{geney}, the scale function $\tilde{S}_n(x)$ 
and the speed measure $\tilde{m}_n(dx)$ associated with 
$Y_n(t, G_n \W)$ are given by 
\begin{align}
\tilde{S}_n(x)&=\frac{2 \tilde{K}_n}{j_n}\int_{-{b_c}}^x
\exp \left( (G_n \W)(y+b_c)-(G_n\W)(b_c) \right) dy,\\
\tilde{m}_n(dx)&=
\tilde{K}_n^{-1}
\exp \left( -\left( (G_n \W)(x+b_c)-(G_n\W)(b_c) \right) \right) 
m_c^{b_c}(dx),
\end{align}
where 
\begin{align}\label{K_n}
\tilde{K}_n &:= \int_{-\varepsilon_0}^{\varepsilon_0} 
\exp \left( 
-\left( (G_n\W)(y+{b_c})-(G_n\W)({b_c}) \right) \right)
m_c^{b_c}(dy), \\ 
\varepsilon_0 &:= \min 
\left\{b_c-\tilde{\ell}_c, \tilde{r}_c-b_c\right\}. \nonumber 
\end{align}
Similarly to \eqref{3_8}, we define a function $\tilde{m}_n(x)$ 
for $x \in \mathbb{R}$ from the measure $\tilde{m}_n(dx)$. 
\begin{align}\label{4_M}
\tilde{M}_n(x) := \tilde{m}_n \circ \left( \tilde{S}_n \right)^{-1}(x) 
= \sup \left\{ \tilde{m}_n(y):\ \tilde{S}_n(y) \leq x \right\}.
\end{align}

For $Q$-almost all $\W \in {\mathbb B}_1$, 
we observe 
\begin{align}\label{ell0}
\lim_{n \to \infty} \tilde{S}_n(x) =  
\begin{cases}
0, & x > \ell_c {-b_c}, \\
-\infty, & x < \ell_c {-b_c}, \\
\end{cases}
\end{align} 
and set $J_1 :=(\ell_c {- b_c}, \infty)$. 
The normalization factor $\tilde{K}_n$ ensures 
that the mass of the speed measure $\tilde{m}_n(dx)$ is 
concentrated near zero. 
Specifically, for $Q$-almost all $\W \in \mathbb{B}_1$, 
the measure $\tilde{m}_n(dx)$ restricted to $J_1$ converges 
vaguely to $\delta_0(dx)$ as $n \to \infty$. 
This concentration implies that the process spends 
the overwhelming majority of its time near the bottom of the valley. 
We also obtain the following convergence. 
\begin{lem}\label{lem1}
On ${\mathbb R}$, the measure 
$\tilde{M}_n(dx)$ 
converges vaguely to $\delta_0(dx)$ as $n \to \infty$ 
for $Q$-almost all $\W \in {\mathbb B}_1$. 
\end{lem}

\begin{proof}
We follow the proof of Lemma 6.1 in \cite{STT}. 
For some $L>0$, we set $x=-b_c+L$. 
{Since $b_c < 0$, the interval $[-b_c, -b_c+L]$ lies 
in the positive region, where the environment is identically zero 
by \eqref{environment}. 
This and $\W(b_c) = V$ imply that 
the increment $(G_n \W)(y+b_c) - (G_n\W)(b_c)=-N^n V$ 
for any $y \in [-b_c, -b_c+L]$. Hence, we obtain} 
\begin{align*}
\tilde{S}_n(x) &= \frac{2 \tilde{K}_n}{j_n} \int_{-b_c}^{-b_c+L} 
\exp\left( (G_n \W)(y+b_c) -(G_n \W)(b_c) \right)\, dy \\
&= \frac{2L \tilde{K}_n}{(rN)^{N^n-2n}} \exp\left(-N^n V \right).
\end{align*}
For an arbitrary fixed $\varepsilon >0$, we set 
\[
x_n:=-b_c + \varepsilon 
j_n^{\log r/\log rN}=-b_c+\varepsilon r^{N^n-2n}, 
\]
which reflects the spatial scaling of the environment. 
Then, we obtain
\begin{align*}
\tilde{S}_n(x_n) &= 2 \varepsilon N^{2n} \tilde{K}_n
\exp\left( -N^n(V+\log N)\right). 
\end{align*} 
{Note that} 
\begin{align*}
\exp \left( 
-\left( (G_n\W)(y+{b_c})-(G_n\W)({b_c}) \right) \right)
 \geq \exp(-N^n \delta),\quad y \in [-\varepsilon_0, \varepsilon_0], 
\end{align*}
{which, together with \eqref{K_n}, implies that} 
\begin{align}\label{tK}
{\tilde{K}_n} &
{\geq \exp(-N^n \delta) 
(m_c^{b_c}(\varepsilon_0)-m_c^{b_c}(-\varepsilon_0))}\nonumber \\
&{=\exp(-N^n \delta) 
(m_c(b_c+\varepsilon_0)-m_c(b_c-\varepsilon_0)).}
\end{align}
{Hence, we have}
\begin{align*}
\tilde{S}_n(x_n) 
\geq 2 \varepsilon N^{2n} \exp \left( -N^n (V+\log N+\delta) \right) 
(m_c(b_c+\varepsilon_0)-m_c(b_c-\varepsilon_0)).
\end{align*}
Since $\delta \in (0, (-\log N -V))$, we obtain that for any $\varepsilon >0$ 
\begin{align}\label{infty}
\lim_{n \to \infty} \tilde{S}_n (x_n)=\infty.
\end{align}

For any arbitrary fixed $\tilde{\varepsilon}_1 <0<\tilde{\varepsilon}_2$, 
we set $\tilde{y}_{n,i}, i=1,2$ 
such that 
\[
\tilde{y}_{n,1}:=\left( \tilde{S}_n \right)^{-1}(\tilde{\varepsilon}_1) 
\quad \text{and}\quad 
\tilde{y}_{n,2}:=\left( \tilde{S}_n \right)^{-1}(\tilde{\varepsilon}_2).
\]
Then, we have 
\begin{align}\label{ell}
\lim_{n \to \infty} \tilde{y}_{n,1}=\ell_c {-b_c}
\quad \text{and}\quad 
\lim_{n \to \infty} \tilde{y}_{n,2}=\infty,
\end{align}
and there exists $n_0 \in {\mathbb N}$ such that for any $n \geq n_0$ 
\begin{align}\label{tilde2}
0=\tilde{S}_n(-b_c)< 
\tilde{\varepsilon}_2 =\tilde{S}_n(\tilde{y}_{n,2})
< \tilde{S}_n \left( -b_c+\varepsilon r^{N^n-2n} \right)
\quad \text{and}\quad 
\tilde{y}_{n,1} < -\varepsilon_0. 
\end{align}
For such a sufficiently large $n$, we divide 
$\tilde{M}_n([\tilde{\varepsilon}_1, \tilde{\varepsilon}_2])$ 
as follows: 
\begin{align*}
\tilde{M}_n([\tilde{\varepsilon}_1, \tilde{\varepsilon}_2]) &= 
\tilde{m}_n([\tilde{y}_{n,1}, \tilde{y}_{n,2}]) \\
&=:k_{n,1}+k_{n,2}+k_{n,3}+k_{n,4},
\end{align*}
where 
\begin{align*}
k_{n,1} &= \tilde{K}_n^{-1} \int_{\tilde{y}_{n,1}}^{-\varepsilon_0}
\exp \left( -N^n(\W(y+b_c)-\W(b_c)) \right) 
m_c^{b_c}(dy), \\
k_{n,2} &= \tilde{K}_n^{-1} \int_{-\varepsilon_0}^{\varepsilon_0} 
\exp \left( -N^n(\W(y+b_c)-\W(b_c)) \right) 
m_c^{b_c}(dy), \\
k_{n,3} &= \tilde{K}_n^{-1} \int_{\varepsilon_0}^{-b_c} 
\exp \left( -N^n(\W(y+b_c)-\W(b_c)) \right) 
m_c^{b_c}(dy), \\
k_{n,4} &= \tilde{K}_n^{-1} \int_{-b_c}^{\tilde{y}_{n,2}} 
\exp \left( -N^n(\W(y+b_c)-\W(b_c)) \right) 
m_c^{b_c}(dy).
\end{align*}
By the definition of $\tilde{K}_n$, we have $k_{n,2}=1$. 
{We next consider the term $k_{n,4}$. 
Since the interval $[-b_c, \tilde{y}_{n,2}]$ 
lies in the positive region, the increment of $\W$ is given by 
\begin{align}\label{CONST}
\W(y+b_c)-\W(b_c)=-V, \quad y \in [-b_c, \tilde{y}_{n,2}]. 
\end{align}
By the monotonicity of $\tilde{S}_n$, the relation \eqref{tilde2} 
implies $\tilde{y}_{n,2} + b_c < \varepsilon r^{N^n-2n}$ 
for a sufficiently large $n$. 
Using the self-similarity \eqref{scale}, we obtain 
\begin{align}\label{scale_m}
m_c(\tilde{y}_{n,2} + b_c) 
&\leq m_c(\varepsilon r^{N^n-2n}) \nonumber \\
&= m_c(\varepsilon) N^{-2n} \exp(N^n \log N).
\end{align}
Furthermore, by \eqref{tK}, 
we have 
\begin{align}\label{tk_1}
\tilde{K}_n^{-1} \leq \exp(N^n \delta) 
(m_c(b_c+\varepsilon_0)-m_c(b_c - \varepsilon_0))^{-1}. 
\end{align}
Combining \eqref{CONST}, \eqref{scale_m}, \eqref{tk_1} and the inequality 
$\log N + V + \delta <0$, 
we can evaluate $k_{n,4}$ as follows:}
\begin{align*}
k_{n,4} 
&= \tilde{K}_n^{-1} \exp \left( N^n V \right) 
\left( m_c^{b_c}(\tilde{y}_{n,2}) - m_c^{b_c}(-b_c) \right) \\
&= \tilde{K}_n^{-1} \exp \left( N^n V \right) m_c(\tilde{y}_{n,2}+b_c) \\
&\leq 
m_c(\varepsilon) N^{-2n} \tilde{K}_n^{-1} \exp( N^n(V+\log N))\\
&\leq \frac{m_c(\varepsilon)}{m_c(b_c+\varepsilon_0)-m_c(b_c-\varepsilon_0)} 
N^{-2n}\exp( N^n(V+\log N+\delta)) \\ 
&\to 0 \quad \text{as } n \to \infty.
\end{align*}
For the remaining terms, the exponential decay mechanism 
(the growth of $\W$ away from the bottom $b_c$) 
ensures their convergence to zero. 
This completes the proof of Lemma \ref{lem1}. 
\end{proof}

The following proposition implies the localization result: 

\begin{prop}\label{dirac}
For $Q$-almost all $\W \in {\mathbb B}_1$, 
the following holds: 
for any integer $d \geq 2, \varepsilon \in (0,1)$ and 
$f_k \in C_0(J_1), k=1, \ldots, d$,
\begin{align}
\lim_{n \to \infty}
E \left[ 
\prod_{k=1}^d 
f_k \left( Y_n(t_k, G_n \W) \right) 
\right] = 
\prod_{k=1}^d \int_{J_1} f_k(y) \delta_0(dy)
\end{align}
uniformly over $(t_1, \ldots, t_d) \in T_{d, \varepsilon}$. 
\end{prop}

\begin{proof}
To show the assertion, we use Theorem 4.1 in \cite{STT}. 
Namely, we check the conditions (A1)--(A4) on page 689 in \cite{STT}. 
The convergence \eqref{ell0} corresponds to condition (A1). 
We remark that for any $n \in {\mathbb N}$ and 
$Q$-almost all $\W \in {\mathbb B}$ 
\[
P\left\{Y_n(0, G_n\W)=-b_c\right\}=1
\quad \text{and}\quad 
-b_c \in J_1,
\]
which satisfies condition (A2). 
As mentioned above, 
the vague convergence of the measure $\tilde{m}_n(dx)$ 
to $\delta_0(dx)$ on $J_1$ corresponds to condition (A3), 
and Lemma \ref{lem1} provides condition (A4). 
We therefore obtain the proposition by applying Theorem 4.1 in \cite{STT}.
\end{proof}

Proposition \ref{dirac} implies that for $Q$-almost all 
$\W \in {\mathbb B}_1$, any integer $d \geq 2$ and $\varepsilon \in (0,1)$, 
\[
\lim_{n \to \infty} P \left\{ \left| X(j_n t_k, G_n \W) -b_c\right|< 
\varepsilon,\ k=1, \ldots, d 
\right\} = 1 
\]
as $n \to \infty$ uniformly over $(t_1, \ldots, t_d) \in T_{d,\varepsilon}$. 
Since the integrand is bounded by $1$, 
the Lebesgue dominated convergence theorem allows us 
to exchange the limit and the integral over ${\mathbb B}_1$ 
with respect to the measure $Q$:
\begin{align}\label{B_1}
\lim_{n \to \infty} \frac{1}{Q({\mathbb B}_1)} 
\int_{{\mathbb B}_1} 
P \left\{ \left| X(j_n t_k, G_n \W) -b_c \right| < \varepsilon,\ 
k=1, \ldots, d 
\right\}\, Q(d \W) = 1
\end{align}
uniformly over $(t_1, \ldots, t_d) \in T_{d,\varepsilon}$.

\subsection{Case of ${\mathbb B}_2$}

We use the following notation: 
\begin{align}\label{a_0}
a_c & \in (b_c, 0) \text{ such that } \W(a_c)=M, \nonumber \\
c_c &:= \sup \left\{ x<a_c:\ \W(x)-\W(b_c)=\log rN \right\}, \nonumber \\
\alpha_c &:= \sup \left\{ x \in (b_c, 0):\ \W(x)=-\log N \right\}.
\end{align}

\subsubsection{The case where $b_c < \alpha_c < c_c < a_c <0$}

Let ${\mathbb B}_2'$ be the subset of $\W \in {\mathbb B}_2$ 
satisfying the condition $b_c < \alpha_c < c_c < a_c <0$. 
For $x \in C$, we set 
\[
\tau(x) := \inf\left\{ t > 0:\ X(t,G_n \W)=x \right\}. 
\]
Then, we have the following: 

\begin{prop}\label{prop1}
For $Q$-almost all $\W \in {\mathbb B}_2'$, 
there exists $\theta \in (0,1)$ 
such that 
\begin{align}
\lim_{n \to \infty} P \left\{ \tau(\alpha_c) < j_n^{\theta} \right\}=1. 
\end{align}
\end{prop}

The proof follows the strategy of Proposition 6.4 (2) 
in \cite{STT}, adapted to our fractal setting.
For $\W \in {\mathbb B}_2'$, we set 
\[
H:= \sup\left\{ \W(x):\ \alpha_c<x<0 \right\} - 
\inf \left\{ \W(x):\ \alpha_c<x<0 \right\}  = M+\log N.
\]
We then choose $\varepsilon' \in (0, (\log r -M)/2)$ and take 
$\eta=\eta(\varepsilon')>0$ such that 
\begin{align}\label{lem_28_1}
H + \varepsilon' < \eta < \log rN-\varepsilon'. 
\end{align}

Next, we define 
\[
\beta_c :=\sup \left\{ 
x <\alpha_c:\ m_c(\alpha_c)-m_c(x) = \varepsilon'
\right\}. 
\]
Since $C$ is unbounded and $m_c(x)$ is 
{the function defined in \eqref{m_c} associated 
with the infinite self-similar measure}, 
such $\beta_c$ is well-defined. 
To analyze the hitting time, we reconstruct 
a modified environment $\overline{W}_c$ from $\W \in {\mathbb B}_2'$ 
by using this $\beta_c$ as follows: 
\begin{align}\label{overline}
\overline{W}_c(x)=
\begin{cases}
\W(x), & x \geq \alpha_c, \\
m_c(x)-(m_c(\beta_c)+\log N + \varepsilon'), & \beta_c<x<\alpha_c, \\
-m_c(x)-(-m_c(\beta_c)+\log N +\varepsilon'), & x \leq \beta_c.
\end{cases}
\end{align}
This modified environment $\overline{W}_c$ is constructed 
to provide a tractable upper bound for the hitting time $\tau(\alpha_c)$ 
in Proposition \ref{prop1}. 
By effectively trapping the process in a deeper auxiliary valley at 
$\beta_c$, we can ensure that the process hits the point $\alpha_c$. 
Using $\overline{W}_c$, we set 
\begin{align*}
\gamma_c &:= \sup\left\{ x < \beta_c:\ 
\overline{W}_c(x)=\overline{W}_c(\alpha_c) \right\}, \\
\lambda_c &:= \sup\left\{ x < \gamma_c:\ 
\overline{W}_c(x)-\overline{W}_c(\beta_c)=\eta \right\}.
\end{align*}
We remark that (i) $\overline{W}_c(\beta_c)=
\min\left\{\overline{W}_c(x):\ x \in {\mathbb R}\right\}$ and 
(ii) $\overline{W}_c(\lambda_c) > M$. 

For a fixed $G_n\overline{W}_c$, we define the shifted process as follows: 
\begin{align}
\bar{Y}_n(t, G_n \overline{W}_c) &:= 
X(j_n^{\eta/\log rN}t, G_n \overline{W}_c) - \beta_c. 
\end{align}
Since $\eta/\log rN <1$ by \eqref{lem_28_1}, 
we can explicitly choose this exponent as the constant $\theta \in (0,1)$ 
required in Proposition \ref{prop1}. 
Note that $\overline{W}_c(x) = \W(x)$ for $x \geq \alpha_c$ 
by \eqref{overline}, which means the process in the modified environment 
$G_n \overline{W}_c$ behaves identically to the one 
in the original environment $G_n \W$ up to the hitting time $\tau(\alpha_c)$. 
Furthermore, since the process starts at $0$ and $\beta_c < \alpha_c < 0$, 
reaching an $\varepsilon$-neighborhood of $\beta_c$ at time $j_n^\theta$ 
guarantees that the continuous path of the process has 
already hit $\alpha_c$. Therefore, the following proposition, 
which asserts the localization of $\bar{Y}_n$ at $\beta_c$, implies 
Proposition \ref{prop1}:

\begin{prop}\label{prop2}
For $Q$-almost all $\overline{W}_c$ and any $\varepsilon>0$, it holds that 
\begin{align}\label{prop27}
\lim_{n \to \infty} 
P \left\{ |\bar{Y}_n(1, G_n \overline{W}_c)| <\varepsilon \right\} =1.
\end{align}
\end{prop}

The scale function $\bar{S}_n(x)$ and the speed measure $\bar{m}_n(dx)$ 
associated with $\bar{Y}_n(t, G_n \overline{W}_c)$ are given by 
\begin{align}
\bar{S}_n(x) &= \frac{2\bar{K}_n}{j_n^{\eta/\log rN}} 
\int_{-\beta_c}^x 
\exp\left( (G_n \overline{W}_c)(y+\beta_c)-(G_n \overline{W}_c)(\beta_c) 
\right)\, dy, \\
\bar{m}_n(dx) &= \bar{K}_n^{-1} \exp\left(- 
\left( (G_n \overline{W}_c)(x+\beta_c)-(G_n \overline{W}_c)(\beta_c)\right) 
\right)\, m_c^{\beta_c}(dx),
\end{align}
where 
\begin{align}
m_c^{\beta_c}(x)&:=m_c(x+\beta_c)-m_c(\beta_c),\\
\bar{K}_n&:= \int_{-\bar{\varepsilon}_0}^{\bar{\varepsilon}_0} 
\exp \left( -N^n 
\left( \overline{W}_c(y+\beta_c)-\overline{W}_c(\beta_c)\right)
\right)\, m_c^{\beta_c}(dy), \\
\bar{\varepsilon}_0 &:= \min\{\beta_c-\gamma_c, \alpha_c-\beta_c\}. 
\nonumber 
\end{align}
Similarly to \eqref{4_M}, we set 
\begin{align}\label{5_M}
\bar{M}_n(x):= \bar{m}_n \circ \left( \bar{S}_n \right)^{-1}(x)
=\sup\left\{ \bar{m}_n(y):\ \bar{S}_n(y) \leq x \right\}.
\end{align}

Since $j_n^{-\eta/\log rN}=\exp(-\eta(N^n-2n))$, 
for $Q$-almost all $\overline{W}_c$, we observe 
\begin{align}\label{A1_1}
\lim_{n \to \infty} \bar{S}_n(x)
= 
\begin{cases}
-\infty, & x < \lambda_c-\beta_c,\\
0, & x > \lambda_c-\beta_c. 
\end{cases}
\end{align}
We set $J_2:=(\lambda_c-\beta_c, \infty)$. 
We remark that the convergence \eqref{A1_1} corresponds to condition (A1) 
in \cite{STT}. 
For $Q$-almost all $\overline{W}_c$, we have that 
$P \left\{
\bar{Y}_n(0, G_n \overline{W}_c)=-\beta_c \right\}=1$ for any 
$n \in {\mathbb N}$ 
and that $-\beta_c \in J_2$. These facts satisfy condition (A2). 
Furthermore, the definition of $\bar{K}_n$ implies that 
on $J_2$, $\bar{m}_n(dx)$ converges vaguely to $\delta_0(dx)$ 
as $n \to \infty$ for $Q$-almost all $\overline{W}_c$, 
which satisfies condition (A3). 
The following lemma corresponds to Lemma \ref{lem1} and 
provides condition (A4). 

\begin{lem}\label{lem2}
On ${\mathbb R}$, the measure 
$\bar{M}_n(dx)$ 
converges vaguely to $\delta_0(dx)$ as $n \to \infty$ 
for $Q$-almost all $\overline{W}_c$. 
\end{lem}

\begin{proof}
We prove this lemma in the same manner as Lemma \ref{lem1}. 
For an arbitrary fixed $\varepsilon >0$, we set 
\[
\bar{x}_n:=-\beta_c+\varepsilon j_n^{\log r/\log rN}
=-\beta_c +\varepsilon r^{N^n-2n}. 
\]
On the interval $[-\beta_c, \bar{x}_n]$, the modified environment 
is identically zero, and 
$\overline{W}_c(\beta_c) = -(\log N + \varepsilon')$ by \eqref{overline}. 
Thus, we have 
\[
(G_n \overline{W}_c)(y+\beta_c) - (G_n \overline{W}_c)(\beta_c)
=N^n(\log N + \varepsilon'),
\quad 
y \in [-\beta_c, \bar{x}_n]. 
\]
Since the length of the integration interval for $\bar{S}_n$ is 
$\varepsilon j_n^{\log r/\log rN}$, we have
\begin{align*}
\bar{S}_n (\bar{x}_n)
&= 2 \varepsilon \bar{K}_n j_n^{(\log r - \eta)/ \log rN} 
\exp( N^n(\log N+\varepsilon') ) \\
&= 2 \varepsilon \bar{K}_n 
\exp(N^n(\log rN - \eta+ \varepsilon'))
\exp(2n(\eta-\log r)).
\end{align*}
In the case of $\overline{W}_c$, instead of \eqref{tK}, 
we have the following inequality: 
\begin{align}\label{tK2}
\bar{K}_n \geq \exp (-N^n \varepsilon') 
\left( 
m_c(\beta_c+\bar{\varepsilon}_0)-m_c(\beta_c-\bar{\varepsilon}_0)
\right).
\end{align}
This implies that for any $\varepsilon >0$
\begin{align*}
\bar{S}_n(\bar{x}_n) 
\geq 2 \varepsilon 
\exp( N^n(\log rN - \eta)) 
\exp( 2n(\eta-\log r)) 
\left( 
m_c(\beta_c+\bar{\varepsilon}_0)-m_c(\beta_c-\bar{\varepsilon}_0)
\right). 
\end{align*}
Since $\log rN - \eta >  \varepsilon'>0$ 
by \eqref{lem_28_1}, we obtain that 
$\lim_{n \to \infty} \bar{S}_n(\bar{x}_n) = \infty$.

For any arbitrary fixed $\bar{\varepsilon}_1<0<\bar{\varepsilon}_2$, 
we set $\bar{y}_{n,i}, i=1,2$ such that 
\[
\bar{y}_{n,1}= \left( \bar{S}_n \right)^{-1}(\bar{\varepsilon}_1)
\quad \text{and} \quad 
\bar{y}_{n,2}= \left( \bar{S}_n \right)^{-1}(\bar{\varepsilon}_2). 
\]
Then, we have
\begin{align}
\lim_{n \to \infty} \bar{y}_{n,1}=\lambda_c -\beta_c 
\quad \text{and} \quad 
\lim_{n \to \infty} \bar{y}_{n,2}=\infty, 
\end{align}
which imply that for all sufficiently large $n$
\begin{align}\label{bar_S}
0=\bar{S}_n(-\beta_c) < \bar{\varepsilon}_2 = \bar{S}_n(\bar{y}_{n,2}) 
< \bar{S}_n(-\beta_c+ \varepsilon r^{N^n-2n})\quad \text{and}\quad 
\bar{y}_{n,1} < -\bar{\varepsilon}_0. 
\end{align}
We set 
\begin{align*}
\bar{M}_n([\bar{\varepsilon}_1, \bar{\varepsilon}_2]) 
&=\bar{m}_n([\bar{y}_{n,1}, \bar{y}_{n,2}]) \\
&=: k_{n,5}+k_{n,6}+k_{n,7}+k_{n,8},
\end{align*}
where 
\begin{align*}
k_{n,5} &= \bar{K}_n^{-1} 
\int_{\bar{y}_{n,1}}^{-\bar{\varepsilon}_0}
\exp \left(
-N^n\{\overline{W}_c(y+\beta_c)-\overline{W}_c(\beta_c)\} 
\right)
m_c^{\beta_c}(dy), \\
k_{n,6} &= \bar{K}_n^{-1} 
\int_{-\bar{\varepsilon}_0}^{\bar{\varepsilon}_0} 
\exp \left(
-N^n\{\overline{W}_c(y+\beta_c)-\overline{W}_c(\beta_c)\}
\right) 
m_c^{\beta_c}(dy), \\
k_{n,7} &= \bar{K}_n^{-1} 
\int_{\bar{\varepsilon}_0}^{-\beta_c} 
\exp \left(
-N^n\{\overline{W}_c(y+\beta_c)-\overline{W}_c(\beta_c)\}
\right) 
m_c^{\beta_c}(dy), \\
k_{n,8} &= \bar{K}_n^{-1} 
\int_{-\beta_c}^{\bar{y}_{n,2}} 
\exp
\left(
-N^n\{\overline{W}_c(y+\beta_c)-\overline{W}_c(\beta_c)\}
\right) 
m_c^{\beta_c}(dy).
\end{align*}
By the definition of $\bar{K}_n$, we have $k_{n,6}=1$. 
The terms $k_{n,5}, k_{n,7}$ and $k_{n,8}$ 
converge to zero as $n \to \infty$ 
for the same reason as for the terms $k_{n,1}, k_{n,3}$ and $k_{n,4}$, 
respectively. 
Specifically, for $k_{n,8}$, 
the increment of $\overline{W}_c$ is given by 
\begin{align}\label{c21}
\overline{W}_c(y+\beta_c)-\overline{W}_c(\beta_c)
= \log N + \varepsilon',\quad 
y \in [-\beta_c, \bar{y}_{n,2}].
\end{align}
The relation \eqref{bar_S} and self-similarity of $m_c$ imply that 
\begin{align}\label{c22}
m_c(\bar{y}_{n,2}+\beta_c) 
&\leq m_c(\varepsilon r^{N^n-2n}) \nonumber \\
&=N^{N^n-2n}m_c(\varepsilon).
\end{align}
Furthermore, by \eqref{tK2}, 
we have 
\begin{align}\label{c23}
\bar{K}_n^{-1} \leq \exp(N^n \varepsilon') 
(m_c(\beta_c+\bar{\varepsilon}_0)-m_c(\beta_c - \bar{\varepsilon}_0))^{-1}. 
\end{align}
Combining \eqref{c21}, \eqref{c22} and \eqref{c23}, 
we can evaluate $k_{n,8}$ as follows:
\begin{align*}
k_{n,8} &= \bar{K}_n^{-1} 
\exp \left(-N^n (\log N + \varepsilon') \right) 
m_c(\bar{y}_{n,2}+\beta_c) \\
&\leq \bar{K}_n^{-1} 
\exp (-N^n \varepsilon') N^{-2n} m_c(\varepsilon) \\
&\leq \frac{m_c(\varepsilon)}
{m_c(\beta_c+\bar{\varepsilon}_0)-m_c(\beta_c - \bar{\varepsilon}_0)} 
N^{-2n} \\ 
&\to 0\ \mbox{as}\ n \to \infty.
\end{align*}
These estimates imply the vague convergence of 
$\bar{M}_n(dx)$ to $\delta_0(dx)$ as $n \to \infty$. 
\end{proof}

\begin{proof}[Proof of Proposition \ref{prop2}] 

The assertion can be shown in the same manner as Proposition \ref{dirac}. 
Specifically, the discussion leading to Lemma \ref{lem2} 
and the lemma itself ensure that the process satisfies 
the four conditions (A1)--(A4) of Theorem 4.1 in \cite{STT}. 
We therefore obtain the localization of 
$\bar{Y}_n(1, G_n \overline{W}_c)$ at the auxiliary valley bottom $\beta_c$, 
which implies Proposition \ref{prop2}. 
\end{proof}

Proposition \ref{prop1} implies that 
for $Q$-almost all $\W \in {\mathbb B}_2'$, 
the process starting at the origin reaches $\alpha_c$ 
within the time scale $j_n^{\theta}$. 

Once the process reaches $\alpha_c$, the problem is reduced 
to the case of $\mathbb{B}_1$. 
Specifically, the process is now driven by the potential gradient 
toward the even deeper valley bottom at $b_c(\W)$ 
within the time scale $j_n$. 
Since the remaining time $j_n - j_n^{\theta}$ is still of order $j_n$, 
we can directly apply the localization result of the case of $\mathbb{B}_1$. 
We thus obtain the following: for any integer $d \geq 2$ and 
$\varepsilon \in (0,1)$, 
\begin{align}\label{remain}
\lim_{n \to \infty} \frac{1}{Q({\mathbb B}_2')} 
\int_{{\mathbb B}_2'} 
P \left\{ \left| X(j_n t_k, G_n \W) -b_c(\W) \right| < \varepsilon, 
k=1,\ldots, d \right\}\, 
Q(d \W) = 1
\end{align}
uniformly over $(t_1, \ldots, t_d) \in T_{d, \varepsilon}$.

\subsubsection{Case of multiple transitions}

Next, we consider the case where $b_c < c_c < \alpha_c < a_c <0$ and 
let ${\mathbb B}_2''$ be the set of environments $\W$ satisfying 
this condition. 
For this case, we set up multiple steps to reach the final localized state. 
For $k \in {\mathbb N}$, we set 
\begin{align}
\alpha_c^{(k)} := \sup \left\{ x<0:\ \W(x)=-k \log N \right\}.
\end{align}
Since $\W(b_c)=V(\log rN) > -\infty$ for $Q$-almost all $\W$, 
the sequence of levels $\{-k \log N\}_{k \in \mathbb{N}}$ eventually 
falls below $V(\log rN)$. 
This ensures the existence of a finite integer $k_0$ such that 
$\alpha_c^{(k_0)}$ is the deepest level satisfying $b_c < \alpha_c^{(k_0)}$. 
Since $c_c < \alpha_c < 0$ for $\W \in {\mathbb B}_2''$, this integer must 
satisfy $k_0 \geq 2$. Consequently, we obtain
\begin{align}\label{k_0}
b_c < \alpha_c^{(k_0)} < \alpha_c^{(k_0-1)} < \dots < \alpha_c^{(1)} 
= \alpha_c < 0. 
\end{align}
Let $X^x(t, G_n\W)$ denote a diffusion process starting at $x$. 
The strong Markov property of the process implies that 
the total travel time can be decomposed into successive hitting times.
For $x$ and $y \in C$, we set 
\[
\tau(x; y) := \inf \left\{t > 0:\ X^x(t,G_n \W)=y \right\}. 
\]
To complete the proof for $\W \in \mathbb{B}_2''$, 
it suffices to show the following proposition:

\begin{prop}\label{prop3}
For $Q$-almost all $\W \in {\mathbb B}_2''$ 
and any integer $k \leq k_0$, 
there exists $\theta \in (0,1)$ such that 
\begin{align}\label{hit_prob}
\lim_{n \to \infty}
P \left\{ \tau(\alpha_c^{(k-1)}; \alpha_c^{(k)}) < j_n^{\theta} \right\}=1. 
\end{align}
\end{prop}

\begin{proof}

For each $k \leq k_0$, we define the local hilltop $M^{(k)}$ 
and its relative height $H^{(k)}$ as follows:
\begin{align*}
M^{(k)} &= \sup\left \{\W(x): \alpha_c^{(k)} < x < \alpha_c^{(k-1)} \right\},\\
H^{(k)} &= \sup\left \{\W(x): \alpha_c^{(k)} < x < \alpha_c^{(k-1)} \right\} 
- \inf\left\{ \W(x): \alpha_c^{(k)} < x < \alpha_c^{(k-1)} \right\} 
=M^{(k)} + k \log N.
\end{align*}
Following the proof of Proposition \ref{prop1}, we choose 
$\varepsilon^{(k)} \in (0, \{\log r-(M^{(k)}+(k-1) \log N)\}/2)$ 
and take a constant $\eta^{(k)}=\eta^{(k)}(\varepsilon^{(k)})$ such that
\begin{align}\label{eta_range}
H^{(k)}+ \varepsilon^{(k)}< \eta^{(k)} < \log rN - \varepsilon^{(k)}. 
\end{align}
Next, we define the auxiliary valley bottom 
\begin{align*}
\beta_c^{(k)} := \sup\left\{x < \alpha_c^{(k)}:\ 
m_c(\alpha_c^{(k)}) - m_c(x) = \varepsilon^{(k)}\right\}, 
\end{align*}
and reconstruct the auxiliary environment $\overline{W}_c^{(k)}$ 
as follows: 
\begin{align}
\overline{W}_c^{(k)}=
\begin{cases}
\W(x),& x \geq \alpha_c^{(k)}, \\
m_c(x)- \left( m_c(\beta_c^{(k)})+k\log N + \varepsilon^{(k)} \right),
& \beta_c^{(k)} < x < \alpha_c^{(k)}, \\ 
-m_c(x)- \left( -m_c(\beta_c^{(k)})+k\log N + \varepsilon^{(k)} \right),
& x \leq \beta_c^{(k)}.
\end{cases}
\end{align}
Then, by the same argument as that for showing 
Proposition \ref{prop2}, we obtain that for $Q$-almost all 
$\overline{W}_c^{(k)}$ and any $\varepsilon >0$ 
\begin{align}\label{step_loc_final}
\lim_{n \to \infty} 
P \left\{ 
\left|
X^{\alpha_c^{(k-1)}}(j_n^{\eta^{(k)}/\log rN} , G_n \overline{W}_c^{(k)})
-\beta_c^{(k)}
\right| < \varepsilon 
\right\} = 1. 
\end{align}
Note that $G_n \overline{W}_c^{(k)} = G_n \W$ on $[\alpha_c^{(k)}, \infty)$. 
Since the point $\alpha_c^{(k)}$ is located between 
the starting point $\alpha_c^{(k-1)}$ and the localization center 
$\beta_c^{(k)}$, the continuity of paths and \eqref{step_loc_final} 
ensure that the process hits $\alpha_c^{(k)}$ 
within the time scale $j_n^{\eta^{(k)}/\log rN}$. 
By setting $\theta = \max_{k \leq k_0} \{ \eta^{(k)} / \log rN \}$, 
the condition (\ref{eta_range}) guarantees that $\theta < 1$. 
Consequently, for any integer $k \leq k_0$, we obtain (\ref{hit_prob}). 
\end{proof}

For $\W \in {\mathbb B}_2''$, we evaluate hitting times for multiple 
steps via Proposition \ref{prop3}. 
Since each time step is extremely small compared to $j_n$, the case 
${\mathbb B}_2''$ reduces to the case ${\mathbb B}_2'$. Namely, 
we have the following: for 
any integer $d \geq 2$ and $\varepsilon \in (0,1)$, 
\begin{align}\label{remain2}
\lim_{n \to \infty} \frac{1}{Q({\mathbb B}_2'')} 
\int_{{\mathbb B}_2''} 
P \left\{ \left| X(j_n t_k, G_n \W) -b_c(\W) \right| < \varepsilon, 
k=1, \ldots, d \right\}\, 
Q(d \W) = 1
\end{align}
uniformly over $(t_1, \ldots, t_d) \in T_{d,\varepsilon}$. 

\begin{proof}[Completion of the Proof of Theorem \ref{t1} (2)] 

We now combine the results for ${\mathbb B}_1$ and 
${\mathbb B}_2$. 

For $Q$-almost all $\W \in {\mathbb B}_1$, the particle is trapped 
in the deep valley $b_c$ throughout the time scale $j_n$ 
as shown in \eqref{B_1}. 

For $Q$-almost all $\W \in {\mathbb B}_2$ (including both cases 
${\mathbb B}_2'$ and ${\mathbb B}_2''$), the particle reaches 
the final localization center $b_c$ after a finite number of transitions 
as shown above. 
Since each of these transitions occurs within a time scale $j_n^{\theta}$ 
that is negligible compared to $j_n$, the remaining time available 
for the process to stay at $b_c$ is still of order $j_n$. 
These arguments ensure that the problem for ${\mathbb B}_2$ 
is effectively reduced 
to the case ${\mathbb B}_1$ once the process hits the point 
$\alpha_c$ in \eqref{a_0} or $\alpha_c^{(k_0)}$ in \eqref{k_0}, 
and we obtain \eqref{remain} and \eqref{remain2}. 

Since the probability is uniformly bounded by 1, 
the Lebesgue dominated convergence theorem 
justifies the exchange of the limit and the integral over ${\mathbb B}$. 
Then, we have that for any integer $d \geq 2$ and $\varepsilon \in (0,1)$, 
\[
\lim_{n \to \infty} \frac{1}{Q({\mathbb B})} \int_{\mathbb B} 
P \left\{ 
\left| X(j_n t_k, G_n \W) -b_c(\W) \right| < \varepsilon, k=1, \ldots, d 
\right\}\, 
Q(d \W) = 1
\]
uniformly over $(t_1, \ldots, t_d) \in T_{d,\varepsilon}$. 
This completes the proof of Theorem \ref{t1} (2). 
\end{proof}

\section*{Acknowledgements}
This work was supported by JSPS KAKENHI Grant Number 24K06786.

\bibliographystyle{elsarticle-num}

\begin{thebibliography}{99}

\bibitem{br} T. Brox, A one-dimensional diffusion process in a Wiener medium, 
Ann. Probab. 14 (1986) 1206--1218.
\url{https://doi.org/10.1214/aop/1176992815}

\bibitem{F1} T. Fujita, A fractional dimension, selfsimilarity and 
a generalized diffusion operator, in: K. Ito, N. Ikeda (Eds.), 
Probabilistic Methods in Mathematical Physics, Proceedings of 
Taniguchi International Symposium, Katata and Kyoto, 1985, 
Kinokuniya, Tokyo, 1987, pp. 83--90.

\bibitem{F2} T. Fujita, Some asymptotic estimates of transition 
probability densities for generalized diffusion processes 
with self-similar speed measures, Publ. Res. Inst. Math. Sci. 26 
(1990) 819--840. 
\url{https://doi.org/10.2977/prims/1195170736}

\bibitem{beyond} A.K. Golmankhaneh, A.S. Balankin, Sub- and super-diffusion 
on Cantor sets: Beyond the paradox, Phys. Lett. A 382 (2018) 960--967. 
\url{https://doi.org/10.1016/j.physleta.2017.12.042}

\bibitem{Golmankhaneh2023} A.K. Golmankhaneh, L.A.O. Ontiveros, 
Fractal calculus approach to diffusion on fractal combs, 
Chaos, Solitons \& Fractals 175 (2023) 113941. 
\url{https://doi.org/10.1016/j.chaos.2023.113941}

\bibitem{IO} M. Iizuka, Y. Ogura, Convergence of one-dimensional diffusion 
processes to a jump process related to population genetics, 
J. Math. Biol. 29 (1991) 671--687. 
\url{https://doi.org/10.1007/BF00164019}

\bibitem{IM} K. Ito, H.P. McKean, Diffusion processes 
and their sample paths, 
Springer-Verlag, Berlin, 1965. 

\bibitem{KS} K. Kawazu, Y. Suzuki, Limit theorems for a diffusion process 
with a one-sided Brownian potential, J. Appl. Probab. 43 (2006) 997--1012.
\url{https://doi.org/10.1239/jap/1167076483}

\bibitem{KST} K. Kawazu, Y. Suzuki, H. Tanaka, A diffusion process 
with a one-sided Brownian potential, Tokyo J. Math. 24 (2001) 211--229.
\url{https://doi.org/10.3836/tjm/1244208535}

\bibitem{Ke} H. Kesten, The limit distribution of Sinai's random walk 
in random environment, Physica A 138 (1986) 299--309. 
\url{https://doi.org/10.1016/0378-4371(86)90186-0}

\bibitem{Og} Y. Ogura, One-dimensional bi-generalized diffusion processes, 
J. Math. Soc. Japan 41 (1989) 213--242.
\url{https://doi.org/10.2969/asjm/04120213}

\bibitem{Si} Ya.G. Sinai, The limit behavior of a one-dimensional random walk 
in a random medium, Theory Probab. Appl. 27 (1982) 256--268. 
\url{https://doi.org/10.1137/1127028}

\bibitem{S} Y. Suzuki, Asymptotic behavior of stochastic processes 
in random environments, Sugaku Expositions 38 (2025) 273--299.
\url{https://doi.org/10.1090/suga/504}

\bibitem{STT} Y. Suzuki, H. Takahashi, Y. Tamura, Diffusion processes 
with one-sided selfsimilar random potentials, Potential Anal. 62 (2025) 
683--701.
\url{https://doi.org/10.1007/s11118-024-10141-8}


\bibitem{TT1} H. Takahashi, Y. Tamura, Homogenization on disconnected 
selfsimilar fractal sets in $\mathbb{R}$, Tokyo J. Math. 28 (2005) 127--138.
\url{https://doi.org/10.3836/tjm/1125603411}

\bibitem{TT2} H. Takahashi, Y. Tamura, Diffusion processes 
in Brownian environments on disconnected selfsimilar fractal sets in 
$\mathbb{R}$, Stat. Probab. Lett. 193 (2023) 109694. 
\url{https://doi.org/10.1016/j.spl.2022.109694}

\bibitem{TTI} T. Takemura, M. Tomisaki, M. Iizuka, 
The weak mutation and strong selection limit of the Moran 
model satisfies the strong Markov property, 
Ann. Reports of Graduate School of Humanities and Sciences. 
Nara Women's University 30 (2015) 105--112. 
\url{https://doi.org/10935/3970}

\bibitem{T} H. Tanaka, Localization of a diffusion process 
in a one-dimensional Brownian environment, Comm. Pure Appl. Math. 47 (1994) 
755--766. 
\url{https://doi.org/10.1002/cpa.3160470602}

\bibitem{layered} V.R. Voller, F.D.A. Aar\~ao Reis, Universal superdiffusive 
infiltration in layered media with fractal distributions of 
low conductivity inclusions, Adv. Water Resour. 172 (2023) 104365. 
\url{https://doi.org/10.1016/j.advwatres.2022.104365}

\end{thebibliography}

\renewcommand{\baselinestretch}{1}

\end{document}